\documentclass[preprint,3p,12pt,times]{elsarticle}

\journal{Elsevier}

\usepackage[T1]{fontenc}
\usepackage{mathrsfs}
\usepackage{amsmath,amssymb}
\newdefinition{definition}{Definition}[section]
\newdefinition{example}{Example}
\newdefinition{remark}{Remark}
\newtheorem{theorem}{Theorem}[section]
\newtheorem{lemma}[theorem]{Lemma}
\newproof{proof}{Proof}
\numberwithin{equation}{section}
\renewcommand{\pi}{\piup}

\renewcommand{\Im}{\operatorname{Im}}
\newcommand{\mathbd}[1]{\boldsymbol{#1}}
\newcommand{\divv}{\mathrm{d}}
\newcommand{\diff}{\,\divv}
\newcommand{\trans}{\mathrm{T}}
\newcommand{\domD}{\mathscr{D}}
\newcommand{\LC}{\mathbf{L}}
\newcommand{\Csp}{\mathbf{C}}
\newcommand{\Ident}{\mathcal{I}}
\newcommand{\Hinf}{\mathbf{H}^{\infty}}
\newcommand{\textSE}{\text{\tiny{\rm SE}}}
\newcommand{\textDE}{\text{\tiny{\rm DE}}}
\newcommand{\SEt}{\psi^{\textSE}}
\newcommand{\DEt}{\psi^{\textDE}}
\newcommand{\SEtInv}{\{\SEt\}^{-1}}
\newcommand{\DEtInv}{\{\DEt\}^{-1}}
\newcommand{\SEtDiv}{\{\SEt\}'}
\newcommand{\DEtDiv}{\{\DEt\}'}
\newcommand{\tSE}{t^{\textSE}}
\newcommand{\tDE}{t^{\textDE}}
\newcommand{\wSE}{w^{\textSE}}
\newcommand{\wDE}{w^{\textDE}}
\newcommand{\Vol}{\mathcal{V}}
\newcommand{\VolSE}{\mathcal{V}^{\textSE}_N}
\newcommand{\VolDE}{\mathcal{V}^{\textDE}_N}
\newcommand{\uSE}{u^{\textSE}_N}
\newcommand{\uDE}{u^{\textDE}_N}
\newcommand{\tuSE}{\tilde{\mathbd{u}}^{\textSE}_n}
\newcommand{\gSE}{G^{\textSE}_N}
\newcommand{\gDE}{G^{\textDE}_N}
\newcommand{\matSE}[1]{{#1}^{\textSE}_n}
\newcommand{\matDE}[1]{{#1}^{\textDE}_n}
\newcommand{\vecSE}[1]{\mathbd{#1}^{\textSE}_n}
\newcommand{\vecDE}[1]{\mathbd{#1}^{\textDE}_n}
\newcommand{\calSE}[1]{\mathcal{#1}^{\textSE}_N}
\newcommand{\calDE}[1]{\mathcal{#1}^{\textDE}_N}
\newcommand{\InZero}{I_n}
\newcommand{\InMinus}{I^{(-1)}_n}
\newcommand{\imnum}{\mathrm{i}\,}
\DeclareMathOperator{\Si}{Si}
\DeclareMathOperator{\arctanh}{arctanh}
\DeclareMathOperator{\arcsinh}{arcsinh}
\DeclareMathOperator{\Order}{O}
\DeclareMathOperator{\rme}{e}
\DeclareMathOperator{\diag}{diag}
\begin{document}

\begin{frontmatter}

\title{Theoretical analysis of a Sinc-Nystr\"{o}m method for Volterra integro-differential equations and its improvement\tnoteref{mytitlenote}}
\tnotetext[mytitlenote]{This work was partially supported by the JSPS Grant-in-Aid for Young Scientists (B) JP17K14147.}

\author{Tomoaki Okayama}
\address{Graduate School of Information Sciences, Hiroshima City University,
3-4-1, Ozuka-higashi, Asaminami-ku, Hiroshima 731-3194, Japan}
\ead{okayama@hiroshima-cu.ac.jp}

 \begin{abstract}
 A Sinc-Nystr\"{o}m method for Volterra integro-differential equations
 was developed by Zarebnia in 2010.
 The method is quite efficient in the sense that
 exponential convergence can be obtained
 even if the given problem has endpoint singularity.
 However, its exponential convergence has not been proved theoretically.
 In addition, to implement
 the method, the regularity of the solution is required, although the solution is an unknown function in practice.
 This paper reinforces the method by presenting two theoretical results:
 1) the regularity of the solution is analyzed, and
 2) its convergence rate is rigorously analyzed.
 Moreover, this paper improves the method
 so that a much higher convergence rate can be attained,
 and theoretical results similar to those listed above are provided.
 Numerical comparisons are also provided.
 \end{abstract}


\end{frontmatter}

\section{Introduction}\label{s:1}

This paper is concerned with
Volterra integro-differential equations of the form
\begin{equation}
\begin{split}
u'(t)&=g(t) + \mu(t) u(t) + \int_a^t k(t,r)u(r)\diff r,\quad a\leq t\leq b,\\
u(a)&=u_a,
\end{split}\label{eq:Volterra-integro-diff}
\end{equation}
where $g(t)$, $\mu(t)$, and $k(t,r)$ are known functions,
and $u(t)$ is the solution to be determined for a given initial value $u_a$.
The equations have been utilized as mathematical models
in many fields, including
population dynamics~\cite{2003Zhao},
finance~\cite{2003Makroglou}, and
viscoelasticity~\cite{2001Shaw}, among others.
Because of their importance in applications,
various numerical methods for solving these equations
have been studied
(see, for example, Brunner~\cite{2004BrunnerBook,2004BrunnerArt},
Driscoll~\cite{driscoll10:_autom},
and the references therein).
Most of these methods seem to assume that
the functions $g(t)$, $\mu(t)$, and $k(t,r)$ are at least continuous
for all $t,\,r \in [a,\,b]$;
otherwise, their convergence becomes poor.


In contrast,
Zarebnia~\cite{2010Zarebnia}
developed quite a promising scheme
by means of the Sinc-Nystr\"{o}m method.
The scheme was derived
without assuming continuity over the whole interval
(e.g., endpoint singularity such as $g(t)=1/\sqrt{t-a}$ is acceptable;
see also Remark~\ref{rem:endpoint-singularity} and Example~\ref{exmp3}).
Furthermore,
its \emph{exponential} convergence, which is much faster than
polynomial convergence,
was suggested in the following way.
The error of the numerical solution $u_N$ was analyzed~\cite{2010Zarebnia} as
\begin{equation}
 \sup_{t\in(a,\,b)}|u(t) - u_N(t)|
\leq C \|A_N^{-1}\|_2 \sqrt{N}\rme^{-\sqrt{\pi d \alpha N}},
\label{leq:Zarebnia-error}
\end{equation}
where $d$ and $\alpha$ are positive parameters that indicate the regularity
of the functions,
and $A_N$ denotes the coefficient matrix of the resulting linear system.
From~\eqref{leq:Zarebnia-error},
we see that the scheme can achieve exponential convergence
if $\|A_N^{-1}\|_2$ does not grow rapidly.
Although
it can be observed in numerical experiments,
no theoretical estimate of $\|A_N^{-1}\|_2$
has yet been given.

The first objective of this paper is to prove
the exponential convergence by showing
\begin{equation}
 \max_{t\in[a,\,b]}|u(t)-u_N(t)|
\leq C \rme^{-\sqrt{\pi d \alpha N}}.
\label{leq:Okayama-error}
\end{equation}
Note that
this approach to the error analysis is completely different
from the one in Zarebnia~\cite{2010Zarebnia};
instead of analyzing the matrix $A_N$,
operator theory is utilized to obtain~\eqref{leq:Okayama-error}.

The second objective of this paper is to
analyze the regularity of solution $u$,
which is important in applications.
In the previous study~\cite{2010Zarebnia},
the regularity of solution $u$ was assumed to be given,
and this was necessary for implementation of the scheme.
In practice, however,
$u$ is an \emph{unknown} function to be determined,
and thus we cannot examine it directly
to investigate its regularity.
To remedy this situation,
this paper shows theoretically that
the necessary information for implementation
(regularity of $u$)
can be determined from the \emph{known} functions $g$, $\mu$,
and $k$.

The third objective of this paper is to improve
the original Sinc-Nystr\"{o}m method
so that it can achieve much faster convergence.
The difference between the original version and our improved version is in the variable transformation;
the single-exponential (SE) transformation is employed
in the original scheme~\cite{2010Zarebnia}
(which is accordingly called the SE-Sinc-Nystr\"{o}m method),
whereas our improved scheme uses the double-exponential (DE) transformation
(which is thus called the DE-Sinc-Nystr\"{o}m method).
In the Sinc numerical method literature,
it is known that
such a replacement generally accelerates
the convergence rate
from $\Order(\exp(-c\sqrt{N}))$ to
$\Order(\exp(-c' N/\log N))$~\cite{2001Mori,2004Sugihara}.
In fact, in this case as well,
error analysis of this paper shows the suggested rate as
\[
  \max_{t\in[a,\,b]}|u(t)-u_N(t)|
\leq C \frac{\log(2 d N/\alpha)}{N} \rme^{-\pi d N/\log(2 d N/\alpha)}.
\]
Furthermore,
regarding the regularity of the solution,
this paper also gives the same theoretical result as above:
the necessary information for implementation
(of the DE-Sinc-Nystr\"{o}m method)
can be determined from the known functions.
Note that we assume the known functions are given in an analytic form;
otherwise, the theoretical result cannot be used.

The remainder of this paper is organized as follows.
In Section~\ref{s:2}, we review the Sinc indefinite integration,
which will be needed in the subsequent sections.
In Section~\ref{sec:SE-Sinc-Nystroem},
the existing results for the SE-Sinc-Nystr\"{o}m method are described, and
we discuss them in terms of the first and second objectives of this paper.
Section~\ref{sec:DE-Sinc-Nystroem} contains
the results on the DE-Sinc-Nystr\"{o}m method (the third objective).
Proofs of the presented theorems are given in Section~\ref{sec:proofs}.
Numerical examples are shown in Section~\ref{sec:numer-exam}.
Concluding remarks are stated in Section~\ref{sec:concl}.

\section{Sinc indefinite integration}\label{s:2}

The Sinc indefinite integration
is an approximation formula for the
indefinite integral of an integrand $F$,
which is defined over the real axis, and expressed as
\begin{equation}
 \int_{-\infty}^{\xi} F(x)\diff x
\approx \sum_{j=-N}^N F(jh) J(j,h)(\xi), \quad \xi\in\mathbb{R}.
\label{approx:Sinc-indef}
\end{equation}
Here, $h$ is the mesh size, and the basis function $J(j,h)(\xi)$ is defined by
\[
 J(j,h)(\xi)
=\int_{-\infty}^{\xi}\frac{\sin[\pi(x/h-j)]}{\pi(x/h-j)}\diff x
=h\left\{\frac{1}{2}+\frac{1}{\pi}\Si[\pi(\xi/h - j)]\right\},
\]
where $\Si(x)=\int_0^x [(\sin \tau)/\tau]\diff \tau$
is the so-called sine integral function.
This formula can be applied
in the case of a finite interval $(a,\,b)$,
by combining it with a variable transformation
that maps $\mathbb{R}$ onto $(a,\,b)$.
Haber~\cite{1993Haber} employed the SE transformation
\[
 s=\SEt(x) = \frac{b-a}{2}\tanh\left(\frac{x}{2}\right)+\frac{b+a}{2},
\]
and
applied~\eqref{approx:Sinc-indef}
with $F(x)=f(\SEt(x))\SEtDiv(x)$ to obtain
\[
\int_a^{t}f(s)\diff s
=\int_{-\infty}^{\SEtInv(t)}f(\SEt(x))\SEtDiv(x)\diff x
\approx\sum_{j=-N}^N f(\tSE_j)\wSE_j(t),
\]
where $\tSE_j=\SEt(jh)$ and $\wSE_j(t)=\SEtDiv(jh)J(j,h)(\SEtInv(t))$.
This approximation is called the SE-Sinc indefinite integration.
Following this, Muhammad--Mori~\cite{2003Muhammad}
proposed replacing the SE transformation
with the DE transformation
\[
 s=\DEt(x) = \frac{b-a}{2}\tanh\left(\frac{\pi}{2}\sinh x\right)+\frac{b+a}{2},
\]
from which they derived the DE-Sinc indefinite integration as
\[
\int_a^{t}f(s)\diff s
=\int_{-\infty}^{\DEtInv(t)}f(\DEt(x))\DEtDiv(x)\diff x
\approx\sum_{j=-N}^N f(\tDE_j)\wDE_j(t),
\]
where $\tDE_j=\DEt(jh)$ and $\wDE_j(t)=\DEtDiv(jh)J(j,h)(\DEtInv(t))$.

These two approximations can achieve exponential convergence.
To describe this more precisely,
we need to introduce the following function space.

\begin{definition}
\label{Def:LC}
Let $\alpha$ be a positive constant
and $\domD$ be a bounded and simply connected domain
(or Riemann surface)
that satisfies $(a,\,b)\subset \domD$.
Then, $\LC_{\alpha}(\domD)$ denotes the family of functions
that are analytic on $\domD$ and bounded by a constant $K$ and the function $Q(z)=(z-a)(b-z)$
for all $z$ in $\domD$ as
\begin{equation}
|f(z)|\leq K |Q(z)|^{\alpha}. \label{Leq:LC-bounded-by-Q}
\end{equation}
\end{definition}

Note that this function space considers
functions of a complex variable,
and hereafter, functions will be supposed to be defined in the complex domain.
In this paper,
the domain $\domD$ is supposed to be either
\begin{equation*}
\SEt(\domD_d) = \{z=\SEt(\zeta):\zeta\in\domD_d\}
\quad \text{or} \quad
\DEt(\domD_d) = \{z=\DEt(\zeta):\zeta\in\domD_d\},
\end{equation*}
which denotes the region translated from the strip domain
$\domD_d=\{\zeta\in\mathbb{C}:|\Im\zeta|<d\}$
for $d>0$.
The former domain $\SEt(\domD_d)$ is a lens-shaped domain,
whereas the latter one $\DEt(\domD_d)$ is an infinitely sheeted
Riemann surface
(see also Tanaka et al.~\cite[Figures~1 and~5]{tanaka09:_desinc}
for the concrete shape of each domain, where
$d=1$ and $(a,\,b)=(-1,\,1)$).
Using these definitions, the convergence theorems
for the SE/DE-Sinc indefinite integration can be stated as follows.

\begin{theorem}[Okayama et al.~{\cite[Theorem 2.9]{okayama13:_error}}]
\label{Thm:SE-Sinc-Indef}
Let $(fQ)\in\LC_{\alpha}(\SEt(\domD_d))$ for $d$ with $0<d<\pi$.
Let $N$ be a positive integer,
and let $h$ be selected by
\begin{equation}
h=\sqrt{\frac{\pi d}{\alpha N}}.\label{Def:h-SE}
\end{equation}
Then, there exists a constant
$C_{\alpha,d}^{\textSE}$ that depends only on $\alpha$ and $d$
such that
\begin{equation*}
   \max_{t\in[a,\,b]}
    \left|\int_{a}^{t}f(s)\diff s
      -\sum_{j=-N}^{N}f(\tSE_j)\wSE_j(t)
    \right|
\leq K (b-a)^{2\alpha - 1} C_{\alpha,d}^{\textSE}
 \rme^{-\sqrt{\pi d \alpha N}},
\end{equation*}
where
$K$ is the constant in~\eqref{Leq:LC-bounded-by-Q}.
\end{theorem}

\begin{theorem}[Okayama et al.~{\cite[Theorem 2.16]{okayama13:_error}}]
\label{Thm:DE-Sinc-Indef}
Let $(fQ)\in\LC_{\alpha}(\DEt(\domD_d))$ for $d$ with $0<d<\pi/2$.
Let $N$ be a positive integer,
and let $h$ be selected by
\begin{equation}
h=\frac{\log(2 d N/\alpha)}{N}.\label{Def:h-DE}
\end{equation}
Then, there exists a constant
$C_{\alpha,d}^{\textDE}$ that depends only on $\alpha$ and $d$
such that
\begin{align*}
\max_{t\in[a,\,b]}
    \left|\int_{a}^{t}f(s)\diff s
      -\sum_{j=-N}^{N}f(\tDE_j)\wDE_j(t)
    \right|
\leq K (b-a)^{2\alpha - 1} C_{\alpha,d}^{\textDE}
  \frac{\log(2 d N/\alpha)}{N}
  \exp\left[\frac{-\pi d N}{\log(2 d N/\alpha)}\right],
\end{align*}
where
$K$ is the constant in~\eqref{Leq:LC-bounded-by-Q}.
\end{theorem}

\begin{remark}
\label{rem:endpoint-singularity}
As mentioned in the introduction,
the assumption $(fQ)\in\LC_{\alpha}(\domD)$
in the theorems does not assume continuity on $[a,\,b]$ overall,
but accepts endpoint singularities.
For example, $g(t)=1/\sqrt{t-a}$ is acceptable
because $(gQ)\in\LC_{1/2}(\domD)$.
\end{remark}

\section{SE-Sinc-Nystr\"{o}m method}
\label{sec:SE-Sinc-Nystroem}

\subsection{Existing results: the proposed scheme and its error analysis}

First, by integrating Eq.~\eqref{eq:Volterra-integro-diff},
we obtain
\begin{equation}
 u(t)=u_a+\int_a^t\left\{
g(s)+\mu(s)u(s) +
\Vol[u](s)
\right\}\diff s,\quad a\leq t\leq b,
\label{eq:real-target}
\end{equation}
where $\Vol[u](s)=\int_a^s k(s,r)u(r)\diff r$.
Zarebnia~\cite{2010Zarebnia}
developed his scheme for~\eqref{eq:real-target}
using the SE-Sinc indefinite integration as follows.
Let $0<\alpha\leq 1$
and let $g Q$, $\mu u Q$, and $(\Vol u) Q$
belong to $\LC_{\alpha}(\SEt(\domD_d))$.
Then, according to Theorem~\ref{Thm:SE-Sinc-Indef},
the integral in~\eqref{eq:real-target}
is approximated as
\begin{align*}
\int_a^t\left\{
g(s)+\mu(s)u(s) +
\Vol[u](s)
\right\}\divv s
\approx
 \sum_{j=-N}^N \{g(\tSE_j) + \mu(\tSE_j)u(\tSE_j)+\Vol[u](\tSE_j)\}\wSE_j(t).
\end{align*}
Furthermore, let
$k(s,\cdot)u(\cdot)Q(\cdot)\in\LC_{\alpha}(\SEt(\domD_d))$
for all $s\in [a,\,b]$.
In the same manner as above, $\Vol u$ is approximated by the term
\[
 \VolSE[u](s)
=\sum_{m=-N}^N k(s,\tSE_m) u(\tSE_m)\wSE_m(s).
\]
With these approximations, we have a new (approximated) equation
\begin{equation}
\uSE(t)=u_a+\sum_{j=-N}^N
\{g(\tSE_j)+\mu(\tSE_j)\uSE(\tSE_j)+\VolSE[\uSE](\tSE_j)\}\wSE_j(t).
\label{eq:SE-Sinc-Nyst-sol}
\end{equation}
The approximated solution $\uSE$ is obtained
if we determine the values
$\vecSE{u}=[\uSE(\tSE_{-N}),\,\uSE(\tSE_{-N+1}),$
$\,\ldots,\,\uSE(\tSE_{N})]^{\trans}$,
where $n=2N+1$.
For this purpose,
we discretize~\eqref{eq:SE-Sinc-Nyst-sol}
at the so-called SE-Sinc points $t=\tSE_i$
$(i=-N,\,\ldots,\,N)$,
which leads to a linear system with respect to $\vecSE{u}$.
Here, let $\matSE{K}$, $\InZero$, and $\InMinus$ be
$n\times n$ matrices whose $(i,\,j)$th elements
$(i,\,j=-N,\,\ldots,\,N)$ are
\[
 (\matSE{K})_{ij}=k(\tSE_i,\tSE_j),
\quad
 (\InZero)_{ij}=\delta_{ij},
\quad
 (\InMinus)_{ij}=\delta_{ij}^{(-1)},
\]
where $\delta_{ij}$ denotes the Kronecker delta, and
$\delta_{ij}^{(-1)}$ is defined by
\[
 \delta_{ij}^{(-1)}=\frac{1}{2} + \frac{1}{\pi}\Si[\pi(i-j)].
\]
Let $\matSE{M}$, $\matSE{D}$, and $\matSE{W}$
be $n\times n$ matrices defined by
\begin{align*}
 \matSE{M}&=\diag[\mu(\tSE_{-N}),\,\ldots,\,\mu(\tSE_N)],\\
 \matSE{D}&=\diag[\SEtDiv(-Nh),\,\ldots,\,\SEtDiv(Nh)],\\
 \matSE{W}&= h\InMinus \matSE{M} \matSE{D}
+ h^2 \InMinus \matSE{D} (\InMinus\circ\matSE{K})\matSE{D},
\end{align*}
where `$\circ$' denotes the Hadamard product.
Then, the linear system to be solved is written in matrix-vector
form as
\begin{equation}
(\InZero - \matSE{W}) \vecSE{u}
=\vecSE{g},
\label{eq:linear-system-SE}
\end{equation}
where $\vecSE{g}$ is an $n$-dimensional vector defined by
\[
 \vecSE{g}=
\left[u_a + \sum_{j=-N}^Ng(\tSE_j)\wSE_j(\tSE_{-N}),\,\ldots,\,
u_a + \sum_{j=-N}^Ng(\tSE_j)\wSE_j(\tSE_N)
\right]^{\trans}.
\]
By solving system~\eqref{eq:linear-system-SE}, $\uSE$ can be determined by
the right-hand side of~\eqref{eq:SE-Sinc-Nyst-sol}.
This is the SE-Sinc-Nystr\"{o}m method derived
by Zarebnia~\cite{2010Zarebnia}.

\begin{remark}
Assume that $\vecSE{u}$ is obtained from~\eqref{eq:linear-system-SE}.
In view of~\eqref{eq:SE-Sinc-Nyst-sol},
one might think that the approximate solution $\uSE$
requires $\Order(n^2)$ to evaluate for each $t$
because it contains the double sum $\sum_j (\VolSE \uSE) \wSE_j$.
However, if vector
$\tuSE=h(\InMinus\circ \matSE{K})\matSE{D}\vecSE{u}$ is computed
before the evaluation,
the term can be computed with $\Order(n)$
by $[\wSE_{-N}(t),\,\ldots,\,\wSE_N(t)]^{\trans}\tuSE$
for each $t$.
\end{remark}

For the SE-Sinc-Nystr\"{o}m method,
the following error analysis was given.

\begin{theorem}[Zarebnia~{\cite[Theorem~2]{2010Zarebnia}}]
\label{thm:Zarebnia}
Let $gQ$, $\mu u Q$, $(\Vol u) Q$ belong to $\LC_{\alpha}(\SEt(\domD_d))$.
Furthermore, let
$k(s,\cdot)u(\cdot)Q(\cdot)\in\LC_{\alpha}(\SEt(\domD_d))$
for all $s\in[a,\,b]$.
Then, there exists a constant $C$ independent of $N$
such that
\begin{equation}
\sup_{t\in(a,\,b)}|u(t)-\uSE(t)|
\leq C \|(\InZero - \matSE{W})^{-1}\|_2 \sqrt{N}\rme^{-\sqrt{\pi d \alpha N}}.
\label{leq:Zarebnia-error-theorem}
\end{equation}
\end{theorem}

\subsection{Two points to be discussed on the existing results}
\label{subsec:two-points}

The first point to be discussed is
the assumptions on solution $u$.
The scheme above is derived
under the assumptions that
$gQ$, $\mu u Q$, $(\Vol u)Q$, and $k(s,\cdot)u(\cdot)Q(\cdot)$
belong to $\LC_{\alpha}(\SEt(\domD_d))$.
In a practical situation, however,
$u$ is an \emph{unknown} function to be solved,
and for this reason, it is impossible to check the assumptions,
at least in a simple way.
Furthermore, as for parameter $d$, the statement
``there exists a constant $d$''
is not sufficient; we need the concrete value of $d$ to launch the scheme.
This is
because $d$ is used in the formula for mesh size $h$ in~\eqref{Def:h-SE}.
Therefore, some sort of remedy is needed
to apply this scheme in practice.

The second point to be discussed is
the solvability and convergence of the scheme.
In~\eqref{leq:Zarebnia-error-theorem},
there exists
the matrix norm of $(\InZero - \matSE{W})^{-1}$,
which clearly depends on $N$.
However, no theoretical estimate of this term
has yet been given.
Therefore, exponential convergence of the scheme is not guaranteed
in a rigorous sense.
In addition, the
invertibility of $(\InZero - \matSE{W})$ is implicitly
assumed in Theorem~\ref{thm:Zarebnia},
but it is not clear and should be proved as part of proving
the scheme's solvability.

\subsection{Theoretical contributions of this paper on these two points}

Let us now introduce the following function space.
\begin{definition}
\label{Def:Hinf}
Let $\domD$ be a bounded and simply connected domain (or Riemann surface).
Then, $\Hinf(\domD)$
denotes the family of functions $f$ that are analytic on $\domD$ and
such that the norm $\|f\|_{\Hinf(\domD)}$ is finite,
where
\[
 \|f\|_{\Hinf(\domD)} = \sup_{z\in\domD}|f(z)|.
\]
\end{definition}

For the first point,
this paper presents the following theorem; the proof is given in Section~\ref{sec:proof-1st}.

\begin{theorem}
\label{thm:rewrite-SE-assump}
Let $gQ$ and $\mu Q$ belong to $\LC_{\alpha}(\SEt(\domD_d))$
for $d$ with $0<d<\pi$.
Moreover, let $k(z,\cdot)Q(\cdot)\in\LC_{\alpha}(\SEt(\domD_d))$
for all $z\in\SEt(\domD_d)$
and let $k(\cdot,w)Q(w)\in\Hinf(\SEt(\domD_d))$
for all $w\in\SEt(\domD_d)$.
Then, all the assumptions
in Theorem~\ref{thm:Zarebnia} are fulfilled.
\end{theorem}

From this theorem, we can see that it is no longer necessary to check the assumptions
on solution $u$;
this is quite a useful result for applications.

For the second point,
this paper presents the following theorem;
the proof is given in Section~\ref{sec:proof-2nd}.

\begin{theorem}
\label{thm:SE-converge}
Let the assumptions
in Theorem~\ref{thm:rewrite-SE-assump} be fulfilled.
Furthermore, let $\mu$,
$k(t,\cdot)Q(\cdot)$, $k(\cdot,s)Q(s)$
belong to $C([a,\,b])$ for all $t\in[a,\,b]$
and $s\in[a,\,b]$.
Then, there exists a positive integer $N_0$
such that for all $N\geq N_0$,
the inverse of $(\InZero - \matSE{W})$ exists,
and there exists a constant $C$ independent of $N$ such that
\[
 \max_{t\in[a,\,b]} |u(t)-\uSE(t)|
\leq C \rme^{-\sqrt{\pi d \alpha N}}.
\]
\end{theorem}

This theorem states the invertibility of matrix $(\InZero-\matSE{W})$,
and it rigorously assures the exponential convergence of $\uSE$.

\section{DE-Sinc-Nystr\"{o}m method}
\label{sec:DE-Sinc-Nystroem}

\subsection{Derivation of the scheme}

The way the DE-Sinc-Nystr\"{o}m method is derived
is quite similar to that for the SE-Sinc-Nystr\"{o}m method.
The important difference between the two is
the variable transformation;
the SE transformation in the previous scheme
is replaced with the DE transformation.

Consider an approximation of the integrals in~\eqref{eq:real-target}
according to Theorem~\ref{Thm:DE-Sinc-Indef}.
Let $0<\alpha\leq 1$,
and let $g Q$, $\mu u Q$, and $(\Vol u) Q$
belong to $\LC_{\alpha}(\DEt(\domD_d))$.
Furthermore, let
$k(s,\cdot)u(\cdot)Q(\cdot)\in\LC_{\alpha}(\DEt(\domD_d))$
for all $s\in [a,\,b]$.
Then, in a similar manner to the SE-Sinc-Nystr\"{o}m method,
we obtain a new equation
\begin{equation}
\uDE(t)=u_a+\sum_{j=-N}^N
\{g(\tDE_j)+\mu(\tDE_j)\uDE(\tDE_j)+\VolDE[\uDE](\tDE_j)\}\wDE_j(t).
\label{eq:DE-Sinc-Nyst-sol}
\end{equation}
The approximated solution $\uDE$ is obtained
if we determine the values
$\vecDE{u}=[\uDE(\tDE_{-N}),\,\uDE(\tDE_{-N+1}),$
$\ldots,\,\uDE(\tDE_{N})]^{\trans}$.
For this purpose,
we discretize~\eqref{eq:DE-Sinc-Nyst-sol}
at the so-called DE-Sinc points $t=\tDE_i$
$(i=-N,\,\ldots,\,N)$,
which leads to a linear system with respect to $\vecDE{u}$.
Here, let $\matDE{K}$ be
an $n\times n$ matrix whose $(i,\,j)$th element
$(i,\,j=-N,\,\ldots,\,N)$ is
\[
 (\matDE{K})_{ij}=k(\tDE_i,\tDE_j).
\]
Let $\matDE{M}$, $\matDE{D}$, and $\matDE{W}$
be $n\times n$ matrices defined by
\begin{align*}
 \matDE{M}&=\diag[\mu(\tDE_{-N}),\,\ldots,\,\mu(\tDE_N)],\\
 \matDE{D}&=\diag[\DEtDiv(-Nh),\,\ldots,\,\DEtDiv(Nh)],\\
 \matDE{W}&= h\InMinus \matDE{M} \matDE{D}
+ h^2 \InMinus \matDE{D} (\InMinus\circ\matDE{K})\matDE{D}.
\end{align*}
Then, the linear system to be solved is written in matrix-vector
form as
\begin{equation}
(\InZero-\matDE{W}) \vecDE{u}
=\vecDE{g},
\label{eq:linear-system-DE}
\end{equation}
where $\vecDE{g}$ is an $n$-dimensional vector defined by
\[
 \vecDE{g}=
\left[u_a + \sum_{j=-N}^Ng(\tDE_j)\wDE_j(\tDE_{-N}),\,\ldots,\,
u_a + \sum_{j=-N}^Ng(\tDE_j)\wDE_j(\tDE_N)
\right]^{\trans}.
\]
By solving system~\eqref{eq:linear-system-DE}, $\uDE$ can be determined by
the right-hand side of~\eqref{eq:DE-Sinc-Nyst-sol}.
This is the DE-Sinc-Nystr\"{o}m method.

\subsection{Theoretical results corresponding to the two points in Section~\ref{sec:SE-Sinc-Nystroem}}

With respect to the first point,
this paper presents the following theorem.
The proof is given in Section~\ref{sec:proof-1st}.

\begin{theorem}
\label{thm:rewrite-DE-assump}
Let $gQ$ and $\mu Q$ belong to $\LC_{\alpha}(\DEt(\domD_d))$
for $d$ with $0<d<\pi/2$.
Moreover, let $k(z,\cdot)Q(\cdot)\in\LC_{\alpha}(\DEt(\domD_d))$
for all $z\in\DEt(\domD_d)$
and $k(\cdot,w)Q(w)\in\Hinf(\DEt(\domD_d))$
for all $w\in\DEt(\domD_d)$.
Then, $g Q$, $\mu u Q$, and $\Vol [u] Q$
belong to $\LC_{\alpha}(\DEt(\domD_d))$.
Furthermore,
$k(s,\cdot)u(\cdot)Q(\cdot)\in\LC_{\alpha}(\DEt(\domD_d))$
for all $s\in [a,\,b]$.
\end{theorem}

With respect to the second point,
this paper presents the following theorem.
The proof is given in Section~\ref{sec:proof-2nd}.

\begin{theorem}
\label{thm:DE-converge}
Let the assumptions
in Theorem~\ref{thm:rewrite-DE-assump} be fulfilled.
Furthermore, let $\mu$,
$k(t,\cdot)Q(\cdot)$, $k(\cdot,s)Q(s)$
belong to $C([a,\,b])$ for all $t\in[a,\,b]$
and $s\in[a,\,b]$.
Then, there exists a positive integer $N_0$
such that for all $N\geq N_0$,
the inverse of $(\InZero-\matDE{W})$ exists,
and there exists a constant $C$ independent of $N$ such that
\[
 \max_{t\in[a,\,b]} |u(t)-\uDE(t)|
\leq C \frac{\log(2 d N/\alpha)}{N}
\rme^{-\pi d N/\log(2 d N/\alpha)}.
\]
\end{theorem}
This theorem states the invertibility of the matrix $(\InZero-\matDE{W})$,
and it can be rigorously shown to have a much higher convergence rate
than the SE-Sinc-Nystr\"{o}m method.

\begin{remark}
\label{rem:condition-number}
Theorems~\ref{thm:SE-converge} and~\ref{thm:DE-converge}
ensure that for all sufficiently large $N$,
$\|(I_n- \matSE{W})^{-1}\|_{\infty}$ and
$\|(I_n- \matDE{W})^{-1}\|_{\infty}$ are finite,
but their \emph{uniform}-boundedness has not been shown yet.
The latter point will be discussed on another occasion.
\end{remark}

\section{Proofs}
\label{sec:proofs}

\subsection{On the first point: Assumptions on the solution}
\label{sec:proof-1st}

The idea behind resolving the first point
(discussed in Section~\ref{subsec:two-points})
is to analyze the regularity of solution $u$
using the following theorem.

\begin{theorem}[Okayama et al.~{\cite[Theorem~3.2]{okayama11:_theor}}]
\label{thm:Vol-sol-regularity}
Consider a Volterra integral equation
\begin{equation}
 u(t) - \int_a^t K(t,s)u(s)\diff s = G(t),\quad a\leq t\leq b.
\label{eq:Volterra-int}
\end{equation}
Let $G\in\Hinf(\domD)$,
let $K(z,\cdot)Q(\cdot)\in\LC_{\alpha}(\domD)$,
and let $K(\cdot,w)Q(w)\in\Hinf(\domD)$,
for all $z,\,w\in\domD$.
Then, the equation~\eqref{eq:Volterra-int} has a unique solution $u\in\Hinf(\domD)$.
\end{theorem}

Notice that
by changing the order of integration,
the equation~\eqref{eq:real-target} can be rewritten as
a Volterra integral equation
\[
 u(t)
-\int_a^t\left(\mu(s)+\int_s^t k(r,s)\diff r\right)u(s)\diff s
=u_a + \int_a^t g(s)\diff s,\quad a\leq t\leq b.
\]
Theorem~\ref{thm:Vol-sol-regularity}
enables us to prove the following theorems.

\begin{theorem}
\label{thm:SE-sol-regularity}
Let the assumptions
of Theorem~\ref{thm:rewrite-SE-assump} be fulfilled.
Then, the equation~\eqref{eq:real-target} has
a unique solution $u\in\Hinf(\SEt(\domD_d))$.
\end{theorem}
\begin{theorem}
\label{thm:DE-sol-regularity}
Let the assumptions
of Theorem~\ref{thm:rewrite-DE-assump} be fulfilled.
Then, the equation~\eqref{eq:real-target} has
a unique solution $u\in\Hinf(\DEt(\domD_d))$.
\end{theorem}

If we prove these theorems, then
Theorems~\ref{thm:rewrite-SE-assump}
and~\ref{thm:rewrite-DE-assump} are established
by the next lemma (and $\LC_1(\domD)\subseteq \LC_{\alpha}(\domD)$
if $\alpha\leq 1$).

\begin{lemma}
Let $\mu Q\in\LC_{\alpha}(\domD)$,
let $k(z,\cdot)Q(\cdot)\in\LC_{\alpha}(\domD)$
for all $z\in\domD$,
and let $k(\cdot,w)Q(w)\in\Hinf(\domD)$.
Furthermore, let $u\in\Hinf(\domD)$.
Then, we have
$\mu u Q\in\LC_{\alpha}(\domD)$,
$(\Vol u)Q\in\LC_{1}(\domD)$,
and $k(s,\cdot)u(\cdot)Q(\cdot)\in\LC_{\alpha}(\domD)$
for all $s\in [a,\,b]$.
\end{lemma}
\begin{proof}
From the assumptions, it is clear that
$\mu u Q$ and $k(s,\cdot)u(\cdot)Q(\cdot)$
belong to $\LC_{\alpha}(\domD)$.
In addition, since $\Vol : \Hinf(\domD)\to\Hinf(\domD)$
(see also Okayama et al.~\cite{okayama11:_theor}),
$\Vol u\in\Hinf(\domD)$ holds,
and as a result, we have $(\Vol u) Q\in\LC_1(\domD)$.
\end{proof}

Thus, it remains
to prove Theorems~\ref{thm:SE-sol-regularity}
and~\ref{thm:DE-sol-regularity}.
Here, let us set
$\tilde{k}(t,s)=\int_s^tk(r,s)\diff r$,
$K(t,s)=\mu(s)+\tilde{k}(t,s)$,
and $G(t)=u_a+\int_a^t g(s)\diff s$.
Then, Theorems~\ref{thm:SE-sol-regularity}
and~\ref{thm:DE-sol-regularity} are proved as follows.
\begin{proof}
Let us show
the assumptions of Theorem~\ref{thm:Vol-sol-regularity}.
First,
notice that
$\int_a^t g(s)\diff s=\Vol g$ in the case $k(t,s)\equiv 1$,
and $\Vol g\in \Hinf(\domD)$ holds.
Therefore, we have $G\in\Hinf(\domD)$.
Next, we consider $K(t,s)$.
It is clear that $\mu Q \in\LC_{\alpha}(\domD) \subset \Hinf(\domD)$.
Finally,
since $k(z,\cdot)Q(\cdot)\in\LC_{\alpha}(\domD)$
and $k(\cdot,w)Q(w)\in\Hinf(\domD)$,
we can see that
$\tilde{k}(z,\cdot)Q(\cdot)\in \LC_{\alpha}(\domD)$
and
$\tilde{k}(\cdot,w)Q(w)\in \Hinf(\domD)$ by observing
\[
 \tilde{k}(z,w)Q(w) = \int_w^z k(r,w)Q(w)\diff r.
\]
This completes the proof.
\end{proof}

\begin{remark}
Theorems~\ref{thm:Vol-sol-regularity} through~\ref{thm:DE-sol-regularity}
present the regularity of solution $u$ assuming
that the given functions belong to a class of analytic functions.
If another class of functions is considered, the result is expected to be
different. See, for example, Kolk et al.~\cite{2009Kolk}
and Pedas--Vainikko~\cite{2009Pedas}.
\end{remark}




\subsection{On the second point: Solvability and convergence}
\label{sec:proof-2nd}

\subsubsection{Solvability of the SE-Sinc-Nystr\"{o}m method}
\label{subsubsec:Solvability-SE}
%
First, consider the SE-Sinc-Nystr\"{o}m method.
Let us write $\Csp=C([a,\,b])$ for short,
and let us define operators
$\mathcal{W}:\Csp\to\Csp$ and $\calSE{W}:\Csp\to\Csp$ as
\begin{align*}
 \mathcal{W} [f](t) &= \int_a^t\left\{\mu(s)f(s)+\Vol[f](s)\right\}\diff s,\\
\calSE{W}[f](t)&=\sum_{j=-N}^N
\left\{\mu(\tSE_j)f(\tSE_j)+\VolSE[f](\tSE_j)\right\}\wSE_j(t).
\end{align*}
Furthermore, let us define a function $\gSE$ (approximation of $G$) as
\[
 \gSE(t) = u_a + \sum_{j=-N}^N g(\tSE_j)\wSE_j(t).
\]
Then, Eqs.~\eqref{eq:real-target}
and~\eqref{eq:SE-Sinc-Nyst-sol}
are written as
\begin{align}
(\Ident - \mathcal{W})u&=G,\label{eq:real-target-symbol}\\
(\Ident - \calSE{W})\uSE&=\gSE.\label{eq:SE-Sinc-symbol}
\end{align}
The invertibility of
$(\InZero - \matSE{W})$ is shown as follows.
The first step is to show that
the equation~\eqref{eq:linear-system-SE}
is uniquely solvable if and only if
the equation~\eqref{eq:SE-Sinc-symbol} is uniquely solvable.
This step is omitted here because
one can easily show it following
Okayama et al.~\cite[Lemma~6.1]{okayama11:_theor}.
The second step is to show
that the equation~\eqref{eq:SE-Sinc-symbol} is uniquely solvable
for all sufficiently large $N$.
This can be shown by
applying the following theorem.

\begin{theorem}[Atkinson~{\cite[Theorem~4.1.1]{atkinson97:_numer}}]
\label{thm:Atkinson-Nystroem}
Assume the following four conditions:\\
 1. Operators $\mathcal{X}$ and $\mathcal{X}_n$
are bounded operators on $\Csp$ to $\Csp$.\\
 2. Operator $(\Ident - \mathcal{X}):\Csp\to\Csp$
has a bounded inverse
$(\Ident - \mathcal{X})^{-1}:\Csp\to\Csp$.\\
 3. Operator $\mathcal{X}_n$ is compact on $\Csp$.\\
 4. The following inequality holds:
\begin{equation*}
\|(\mathcal{X}-\mathcal{X}_n)\mathcal{X}_n\|_{\mathcal{L}(\Csp,\Csp)}
<\frac{1}{\|(\Ident - \mathcal{X})^{-1}\|_{\mathcal{L}(\Csp,\Csp)}}.
\end{equation*}
Then, $(\Ident - \mathcal{X}_n)^{-1}$ exists as a bounded operator
on $\Csp$ to $\Csp$, with
\begin{equation}
\|(\Ident - \mathcal{X}_n)^{-1}\|_{\mathcal{L}(\Csp,\Csp)}
\leq \frac{1 + \|(\Ident - \mathcal{X})^{-1}\|_{\mathcal{L}(\Csp,\Csp)}
\|\mathcal{X}_n\|_{\mathcal{L}(\Csp,\Csp)}}
{1 - \|(\Ident - \mathcal{X})^{-1}\|_{\mathcal{L}(\Csp,\Csp)}\|(\mathcal{X}-\mathcal{X}_n)\mathcal{X}_n\|_{\mathcal{L}(\Csp,\Csp)}}.
\label{InEq:Bound-Inverse-Op}
\end{equation}
\end{theorem}

In what follows, we show that the four conditions
of Theorem~\ref{thm:Atkinson-Nystroem}
are fulfilled
with $\mathcal{X}=\mathcal{W}$ and $\mathcal{X}_n = \calSE{W}$,
under the assumptions of Theorem~\ref{thm:SE-converge}.
Condition~1 clearly holds.
Condition~2 is a classical result.
Condition~3 immediately follows from the Arzel\`{a}--Ascoli theorem.
The most difficult task is showing condition~4.
For this purpose,
we need a bound on the basis function
$J(j,h)(x)$, as follows.
\begin{lemma}[Stenger~{\cite[Lemma~3.6.5]{stenger93:_numer}}]
\label{Lem:Bound-J-Real}
For all $x\in\mathbb{R}$, it holds that
\begin{equation*}
|J(j,h)(x)|
\leq 1.1 h.
\end{equation*}
\end{lemma}
%
%
\begin{lemma}[Okayama et al.~{\cite[Lemma~6.4]{okayama11:_theor}}]
\label{Lem:Bound-J-Complex}
For all $x\in\mathbb{R}$ and $y\in\mathbb{R}$, it holds that
\begin{equation*}
|J(j,h)(x+\imnum y)|
\leq \frac{5h}{\pi}\cdot\frac{\sinh(\pi y/h)}{\pi y/h}.
\end{equation*}
\end{lemma}

Using this lemma, we can prove
the convergence of the term
$\|(\mathcal{W} - \calSE{W})\calSE{W}\|_{\mathcal{L}(\Csp,\Csp)}$
as described below.

\begin{lemma}
\label{Lem:SE-Sinc-converge-main}
Let $\mu$ and $k$ satisfy the assumptions
in Theorem~\ref{thm:SE-converge}.
Then, there exists a constant $C$ independent of $N$
such that
\begin{equation*}
\|(\mathcal{W} - \calSE{W})\calSE{W}\|_{\mathcal{L}(\Csp,\Csp)}
\leq C h,
\end{equation*}
where $h$ is the mesh size defined by~\eqref{Def:h-SE}.
\end{lemma}
\begin{proof}
We show that there exists a constant $C$
independent of $f$ and $N$ such that
\[
 |(\mathcal{W}-\calSE{W})\calSE{W}[f](t)|
\leq C \|f\|_{\Csp} h.
\]
Let us define functions $F_j(s)$ and $E_j(s)$ as
\begin{align*}
 F_j(s)&=\mu(s) J(j,h)(\SEtInv(s))
+\int_a^s k(s,r)J(j,h)(\SEtInv(r))\diff r,\\
 E_j(t)&=\int_a^t F_j(s)\diff s -\sum_{k=-N}^N F_j(\tSE_k)\wSE_k(t).
\end{align*}
Then, we have
\begin{align}
(\mathcal{W}-\calSE{W})\calSE{W}[f](t)
&=\sum_{j=-N}^N \mu(\tSE_j)f(\tSE_j)\SEtDiv(jh) E_j(t)\nonumber\\
&\quad
+\sum_{j=-N}^N\SEtDiv(jh)E_j(t)
\sum_{k=-N}^N k(\tSE_j,\tSE_k) f(\tSE_k)\wSE_k(\SEt(jh)).
\label{eq:operator-error-SE}
\end{align}
Since $h\sum_j |\mu(\tSE_j)|\SEtDiv(jh)$
converges to $\int_a^b |\mu(s)|\diff s$ as $N\to\infty$,
the first term of~\eqref{eq:operator-error-SE} is bounded as
\begin{align*}
 \left|
\sum_{j=-N}^N \mu(\tSE_j)f(\tSE_j)\SEtDiv(jh) E_j(t)
\right|
\leq
\|f\|_{\Csp}
\max_j|E_j(t)|
\left\{
\sum_{j=-N}^N |\mu(\tSE_j)|\SEtDiv(jh)
\right\}
\leq \|f\|_{\Csp} \max_j|E_j(t)| \frac{C_1}{h}
\end{align*}
for some constant $C_1$ independent of $f$ and $N$.
Similarly, from the convergence
\[
 h^2 \sum_{j=-N}^N\sum_{k=-N}^N |k(\tSE_j,\tSE_k)|\SEtDiv(jh)\SEtDiv(kh)
\to \int_a^b \left(\int_a^b |k(t,s)|\diff s\right)\diff t
\]
as $N\to\infty$,
the second term of~\eqref{eq:operator-error-SE} is bounded as
\begin{align*}
\left|
\sum_{j=-N}^N\SEtDiv(jh)E_j(t)
\sum_{k=-N}^N k(\tSE_j,\tSE_k) f(\tSE_k)\wSE_k(\SEt(jh))
\right|
\leq \|f\|_{\Csp}\max_j|E_j(t)|\frac{C_2}{h}
\end{align*}
for some constant $C_2$ independent of $f$ and $N$
(note that $|\wSE_k(\SEt(jh))|\leq 1.1h\cdot\SEtDiv(kh)$ holds
by Lemma~\ref{Lem:Bound-J-Real}).
What is left is to bound $|E_j(t)|$.
By the assumptions on $\mu$ and $k$,
there exist constants $C_3$ and $C_4$
independent of $f$ and $N$ such that
\begin{equation}
|\mu(z)Q(z)|
\leq C_3|Q(z)|^{\alpha},
\quad
\int_a^z|k(z,w)||\divv w|
\leq C_4.
\label{leq:mu-k-bound}
\end{equation}
From this and Lemma~\ref{Lem:Bound-J-Complex}, it holds that
\begin{align*}
 |F_j(z)Q(z)|
\leq \left(C_3 |Q(z)|^{\alpha}
+ C_4 |Q(z)|^{\alpha}|Q(z)|^{1-\alpha}\right)\max_j|J(j,h)(\SEtInv(z))|
\leq C_5 |Q(z)|^{\alpha} \frac{5h}{\pi}\frac{\sinh(\pi d/h)}{\pi d/h}
\end{align*}
for some constant $C_5$ independent of $f$ and $N$.
Therefore, $F_j$ satisfies
the assumptions of Theorem~\ref{Thm:SE-Sinc-Indef},
from which
we have
\begin{align*}
|E_j(t)|\leq
C_5 \frac{5h}{\pi}\frac{\sinh(\pi d/h)}{\pi d/h}
(b-a)^{2\alpha-1}
C_{\alpha,d}^{\textSE}
 \rme^{-\sqrt{\pi d \alpha N}}
=\frac{5C_5C_{\alpha,d}^{\textSE}}{\pi^2 d}
\left[\sinh(\pi d/h)\rme^{-\pi d/h}\right] h^2
\leq \frac{5C_5C_{\alpha,d}^{\textSE}}{2\pi^2 d} h^2.
\end{align*}
Summing up the above results, we finally have
\begin{align*}
|(\mathcal{W}-\calSE{W})\calSE{W}[f](t)|
\leq \|f\|_{\Csp}
\left\{\frac{C_1}{h} + \frac{C_2}{h}\right\}\max_{j}|E_j(t)|
\leq\|f\|_{\Csp}
\left\{C_1 + C_2\right\}
\frac{5C_5C_{\alpha,d}^{\textSE}}{2\pi^2 d} h,
\end{align*}
which is the desired inequality.
\end{proof}

Thus, condition~4 in
Theorem~\ref{thm:Atkinson-Nystroem} is fulfilled
for all sufficiently large $N$.
As a result,
$(\Ident -\calSE{W})$ has a bounded inverse,
and so the equation~\eqref{eq:SE-Sinc-symbol}
is uniquely solvable.
This shows the existence of $(\InZero - \matSE{W})^{-1}$
as was previously explained.
In summary, the next lemma holds.
\begin{lemma}
\label{lem:SE-inverse-exist}
Let the assumptions of Theorem~\ref{thm:SE-converge}
be fulfilled. Then,
there exists a positive integer $N_0$
such that for all $N\geq N_0$,
$(\InZero-\matSE{W})^{-1}:\mathbb{R}^{n}\to\mathbb{R}^n$
and $(\Ident - \calSE{W})^{-1}:\Csp\to\Csp$ exist,
and it holds that
\[
 \|u - \uSE\|_{\Csp}
\leq \|(\Ident - \calSE{W})^{-1}\|_{\mathcal{L}(\Csp,\Csp)}
 \left\{\|\mathcal{W}u -\calSE{W}u\|_{\Csp} + \|G - \gSE\|_{\Csp}\right\}.
\]
Furthermore,
there exists a constant $C$ independent of $N$ such that
\[
 \|(\Ident - \calSE{W})^{-1}\|_{\mathcal{L}(\Csp,\Csp)}\leq C.
\]
\end{lemma}
\begin{proof}
Using~\eqref{eq:real-target-symbol},~\eqref{eq:SE-Sinc-symbol},
and the existence of $(\Ident - \calSE{W})^{-1}$, we have
\begin{align}
u - \uSE
&=u -
(\Ident - \calSE{W})^{-1} G
+(\Ident - \calSE{W})^{-1} G
- (\Ident - \calSE{W})^{-1} \gSE\nonumber\\
&=(\Ident - \calSE{W})^{-1}
\left\{(\Ident - \calSE{W})u - G\right\}
+(\Ident - \calSE{W})^{-1}(G - \gSE)\nonumber\\
&=(\Ident - \calSE{W})^{-1}
\left\{(\mathcal{W}u -\calSE{W}u) + (G - \gSE)\right\}.\nonumber
\end{align}
The proof is completed by showing the
boundedness of
$\|(\Ident - \calSE{W})^{-1}\|_{\mathcal{L}(\Csp,\Csp)}$.
From inequality~\eqref{InEq:Bound-Inverse-Op},
it holds that
\[
 \|(\Ident - \calSE{W})^{-1}\|_{\mathcal{L}(\Csp,\Csp)}
\leq \frac{1 + \|(\Ident - \mathcal{W})^{-1}\|_{\mathcal{L}(\Csp,\Csp)}
\|\calSE{W}\|_{\mathcal{L}(\Csp,\Csp)}}
{1 - \|(\Ident - \mathcal{W})^{-1}\|_{\mathcal{L}(\Csp,\Csp)}\|(\mathcal{W}-\calSE{W})\calSE{W}\|_{\mathcal{L}(\Csp,\Csp)}}.
\]
Since
$\|(\Ident - \mathcal{W})^{-1}\|_{\mathcal{L}(\Csp,\Csp)}$
is a constant
and $\|(\mathcal{W}-\calSE{W})\calSE{W}\|_{\mathcal{L}(\Csp,\Csp)}\to 0$
as $N\to\infty$,
it remains to show the
boundedness of $\|\calSE{W}\|_{\mathcal{L}(\Csp,\Csp)}$.
First,
\begin{align*}
 \calSE{W}[f](t)
=\sum_{j=-N}^N \mu(\tSE_j)f(\tSE_j)\wSE_j(t)
+\sum_{j=-N}^N \wSE_j(t)
\sum_{k=-N}^N k(\tSE_j,\tSE_k) f(\tSE_k)\wSE_k(\SEt(jh)),
\end{align*}
which is quite similar to~\eqref{eq:operator-error-SE}.
The estimate proceeds in a similar manner, and as a result,
it holds that
\[
|\calSE{W}[f](t)|
\leq \|f\|_{\Csp} \left\{\frac{C_1}{h}+\frac{C_2}{h}\right\}(1.1 h)
=\|f\|_{\Csp} 1.1\left\{C_1 + C_2\right\}
\]
for the same constants $C_1$ and $C_2$ as before.
This completes the proof.
\end{proof}

\subsubsection{Convergence of the SE-Sinc-Nystr\"{o}m method}
\label{subsubsec:Convergence-SE}

Thanks to Lemma~\ref{lem:SE-inverse-exist},
Theorem~\ref{thm:SE-converge} is established
if the next lemma is proved.

\begin{lemma}
Let the assumptions of Theorem~\ref{thm:SE-converge}
be fulfilled,
and let $N_0$ be
the positive integer appearing in Lemma~\ref{lem:SE-inverse-exist}.
Then, there exist constants $\tilde{C}_1$ and $\tilde{C}_2$
independent of $N$ such that for all $N\geq N_0$,
\begin{align}
\|\mathcal{W} u - \calSE{W} u\|_{\Csp}
&\leq \tilde{C}_1 \rme^{-\sqrt{\pi d \alpha N}},\label{leq:calSE-Wu}\\
\|G - \gSE\|_{\Csp}
&\leq \tilde{C}_2 \rme^{-\sqrt{\pi d \alpha N}}.\label{leq:G-SE-indef}
\end{align}
\end{lemma}
\begin{proof}
Since $gQ\in\LC_{\alpha}(\SEt(\domD_d))$,
\eqref{leq:G-SE-indef} clearly holds
from Theorem~\ref{Thm:SE-Sinc-Indef}.
For~\eqref{leq:calSE-Wu},
it holds that
\begin{align*}
&\mathcal{W}[u](t) - \calSE{W}[u](t)\\
&=\left[\int_a^t \mu(s)u(s)\diff s
-\sum_{j=-N}^N \mu(\tSE_j)u(\tSE_j)\wSE_j(t)
\right]\\
&\quad +
\left[\int_a^t\Vol[u](s)\diff s
-\sum_{j=-N}^N \Vol[u](\tSE_j)\wSE_j(t)
\right]
+\left[
\sum_{j=-N}^N \left\{\Vol[u](\tSE_j) - \VolSE[u](\tSE_j) \right\}\wSE_j(t)
\right].
\end{align*}
Since Theorem~\ref{thm:rewrite-SE-assump}
claims
$\mu u Q$ and $(\Vol u)Q$
belong to $\LC_{\alpha}(\SEt(\domD_d))$,
the first and second terms are bounded
using Theorem~\ref{Thm:SE-Sinc-Indef}
as in~\eqref{leq:G-SE-indef}.
For the third term, since
$k(s,\cdot)u(\cdot)Q(\cdot)\in\LC_{\alpha}(\SEt(\domD_d))$
from Theorem~\ref{thm:rewrite-SE-assump},
$|\Vol[u](\tSE_j) - \VolSE[u](\tSE_j)|$
is bounded using Theorem~\ref{Thm:SE-Sinc-Indef}.
Using this bound, we have
\begin{align*}
\left|
\sum_{j=-N}^N \left\{\Vol[u](\tSE_j) - \VolSE[u](\tSE_j) \right\}
\cdot \wSE_j(t)
\right|
&\leq \sum_{j=-N}^N
\left(
\tilde{C}_5 \rme^{-\sqrt{\pi d \alpha N}}
\right)
\cdot |\wSE_j(t)|
\end{align*}
for some constant $\tilde{C}_5$.
In addition,
since
$h\sum_{j=-N}^N\SEtDiv(jh)$ converges to $(b-a)$
as $N\to\infty$,
and from Lemma~\ref{Lem:Bound-J-Real},
there exists a constant $\tilde{C}_6$ such that
\[
\sum_{j=-N}^N|\wSE_j(t)|
\leq \sum_{j=-N}^N (1.1 h \SEtDiv(jh))
\leq \tilde{C}_6.
\]
This completes the proof.
\end{proof}

\subsubsection{Solvability of the DE-Sinc-Nystr\"{o}m method}

We proceed now to the case of the DE-Sinc-Nystr\"{o}m method.
Let us introduce an operator $\calDE{W}:\Csp\to\Csp$
and a function $\gDE$ as
\begin{align*}
\calDE{W}[f](t)&=\sum_{j=-N}^N
\left\{\mu(\tDE_j)f(\tDE_j)+\VolDE[f](\tDE_j)\right\}\wDE_j(t),\\
 \gDE(t) &= u_a + \sum_{j=-N}^N g(\tDE_j)\wDE_j(t).
\end{align*}
The proof proceeds in the same manner as in the SE case
(Section~\ref{subsubsec:Solvability-SE}).
First, the four conditions in
Theorem~\ref{thm:Atkinson-Nystroem} are confirmed
with $\mathcal{X}=\mathcal{W}$ and $\mathcal{X}_n=\calDE{W}$.
Conditions~1 through~3 are shown in the same way.
Condition~4 is shown as follows.
\begin{lemma}
\label{Lem:DE-Sinc-converge-main}
Let $\mu$ and $k$ satisfy the assumptions
in Theorem~\ref{thm:DE-converge}.
Then, there exists a constant $C$ independent of $N$
such that
\begin{equation*}
\|(\mathcal{W} - \calDE{W})\calDE{W}\|_{\mathcal{L}(\Csp,\Csp)}
\leq C h^2,
\end{equation*}
where $h$ is the mesh size defined by~\eqref{Def:h-DE}.
\end{lemma}
\begin{proof}
We show that there exists a constant $C$
independent of $f$ and $N$ such that
\[
 |(\mathcal{W}-\calDE{W})\calDE{W}[f](t)|
\leq C \|f\|_{\Csp} h^2.
\]
Let us define functions $F_j(s)$ and $E_j(s)$ as
\begin{align*}
 F_j(s)&=\mu(s) J(j,h)(\DEtInv(s))
+\int_a^s k(s,r)J(j,h)(\DEtInv(r))\diff r,\\
 E_j(t)&=\int_a^t F_j(s)\diff s -\sum_{k=-N}^N F_j(\tDE_k)\wDE_k(t).
\end{align*}
Then, we have
\begin{align}
(\mathcal{W}-\calDE{W})\calDE{W}[f](t)
&=\sum_{j=-N}^N \mu(\tDE_j)f(\tDE_j)\DEtDiv(jh) E_j(t)\nonumber\\
&\quad +
\sum_{j=-N}^N\DEtDiv(jh)E_j(t)
\sum_{k=-N}^N k(\tDE_j,\tDE_k) f(\tDE_k)\wDE_k(\DEt(jh)).
\label{eq:operator-error-DE}
\end{align}
Since $h\sum_j |\mu(\tDE_j)|\DEtDiv(jh)$
converges to $\int_a^b |\mu(s)|\diff s$ as $N\to\infty$,
the first term of~\eqref{eq:operator-error-DE} is bounded as
\begin{align*}
 \left|
\sum_{j=-N}^N \mu(\tDE_j)f(\tDE_j)\DEtDiv(jh) E_j(t)
\right|
\leq
\|f\|_{\Csp}
\max_j|E_j(t)|
\left\{
\sum_{j=-N}^N |\mu(\tDE_j)|\DEtDiv(jh)
\right\}
\leq \|f\|_{\Csp} \max_j|E_j(t)| \frac{C_1}{h}
\end{align*}
for some constant $C_1$ independent of $f$ and $N$.
Similarly, from the convergence
\[
 h^2 \sum_{j=-N}^N\sum_{k=-N}^N |k(\tDE_j,\tDE_k)|\DEtDiv(jh)\DEtDiv(kh)
\to \int_a^b \left(\int_a^b |k(t,s)|\diff s\right)\diff t
\]
as $N\to\infty$,
the second term of~\eqref{eq:operator-error-DE} is bounded as
\begin{align*}
\left|
\sum_{j=-N}^N\DEtDiv(jh)E_j(t)
\sum_{k=-N}^N k(\tDE_j,\tDE_k) f(\tDE_k)\wDE_k(\DEt(jh))
\right|
\leq \|f\|_{\Csp}\max_j|E_j(t)|\frac{C_2}{h}
\end{align*}
for some constant $C_2$ independent of $f$ and $N$
(note that $|\wDE_k(\DEt(jh))|\leq 1.1h\cdot\DEtDiv(kh)$ holds
by Lemma~\ref{Lem:Bound-J-Real}).
What is left is to bound $|E_j(t)|$.
By the assumptions on $\mu$ and $k$,
there exist constants $C_3$ and $C_4$
independent of $f$ and $N$ such that~\eqref{leq:mu-k-bound} holds.
From this and Lemma~\ref{Lem:Bound-J-Complex}, it holds that
\begin{align*}
 |F_j(z)Q(z)|
\leq \left(C_3 |Q(z)|^{\alpha}
+ C_4 |Q(z)|^{\alpha}|Q(z)|^{1-\alpha}\right)\max_j|J(j,h)(\DEtInv(z))|
\leq C_5 |Q(z)|^{\alpha} \frac{5h}{\pi}\frac{\sinh(\pi d/h)}{\pi d/h},
\end{align*}
for some constant $C_5$ independent of $f$ and $N$.
Therefore, $F_j$ satisfies
the assumptions of Theorem~\ref{Thm:DE-Sinc-Indef},
from which
we have
\begin{align*}
|E_j(t)|&\leq
C_5 \frac{5h}{\pi}\frac{\sinh(\pi d/h)}{\pi d/h}
(b-a)^{2\alpha-1}
C_{\alpha,d}^{\textDE}
 \frac{\log(2 d N/\alpha)}{N}\rme^{-\pi d N/\log(2 d N/\alpha)}\\
&=\frac{5C_5C_{\alpha,d}^{\textDE}}{\pi^2 d}
\left[\sinh(\pi d/h)\rme^{-\pi d/h}\right] h^3
\leq \frac{5C_5C_{\alpha,d}^{\textDE}}{2\pi^2 d} h^3.
\end{align*}
Summing up the above results, we finally have
\begin{align*}
|(\mathcal{W}-\calDE{W})\calDE{W}[f](t)|
\leq \|f\|_{\Csp}
\left\{\frac{C_1}{h} + \frac{C_2}{h}\right\}\max_{j}|E_j(t)|
\leq\|f\|_{\Csp}
\left\{C_1 + C_2\right\}
\frac{5C_5C_{\alpha,d}^{\textDE}}{2\pi^2 d} h^2,
\end{align*}
which is the desired inequality.
\end{proof}

Thus, condition~4 in
Theorem~\ref{thm:Atkinson-Nystroem} is fulfilled
for all sufficiently large $N$.
As a summary of this part,
the next lemma holds. The proof is omitted because
it proceeds in the same way as the proof of Lemma~\ref{lem:SE-inverse-exist}.
\begin{lemma}
\label{lem:DE-inverse-exist}
Let the assumptions of Theorem~\ref{thm:DE-converge}
be fulfilled. Then,
there exists a positive integer $N_0$
such that for all $N\geq N_0$,
$(\InZero-\matDE{W})^{-1}:\mathbb{R}^{n}\to\mathbb{R}^n$
and $(\Ident - \calDE{W})^{-1}:\Csp\to\Csp$ exist,
and it holds that
\[
 \|u - \uDE\|_{\Csp}
\leq \|(\Ident - \calDE{W})^{-1}\|_{\mathcal{L}(\Csp,\Csp)}
 \left\{\|\mathcal{W}u -\calDE{W}u\|_{\Csp} + \|G - \gDE\|_{\Csp}\right\}.
\]
Furthermore,
there exists a constant $C$ independent of $N$ such that
\[
 \|(\Ident - \calDE{W})^{-1}\|_{\mathcal{L}(\Csp,\Csp)}\leq C.
\]
\end{lemma}

\subsubsection{Convergence of the DE-Sinc-Nystr\"{o}m method}

Thanks to Lemma~\ref{lem:DE-inverse-exist},
Theorem~\ref{thm:DE-converge} is established
if the next lemma is proved.

\begin{lemma}
Let the assumptions of Theorem~\ref{thm:DE-converge}
be fulfilled,
and let $N_0$ be
the positive integer appearing in Lemma~\ref{lem:DE-inverse-exist}.
Then, there exist constants $\tilde{C}_1$ and $\tilde{C}_2$
independent of $N$ such that for all $N\geq N_0$,
\begin{align}
\|\mathcal{W} u - \calDE{W} u\|_{\Csp}
&\leq \tilde{C}_1
\frac{\log(2 d N/\alpha)}{N}
 \rme^{-\pi d N/\log(2 d N/\alpha)},\label{leq:calDE-Wu}\\
\|G - \gDE\|_{\Csp}
&\leq \tilde{C}_2
\frac{\log(2 d N/\alpha)}{N}
 \rme^{-\pi d N/\log(2 d N/\alpha)}.\label{leq:G-DE-indef}
\end{align}
\end{lemma}
\begin{proof}
Since $gQ\in\LC_{\alpha}(\DEt(\domD_d))$,
\eqref{leq:G-DE-indef} clearly holds
from Theorem~\ref{Thm:DE-Sinc-Indef}.
For~\eqref{leq:calDE-Wu},
it holds that
\begin{align*}
& \mathcal{W}[u](t) - \calDE{W}[u](t)\\
&=\left[\int_a^t \mu(s)u(s)\diff s
-\sum_{j=-N}^N \mu(\tDE_j)u(\tDE_j)\wDE_j(t)
\right]\\
&\quad +
\left[\int_a^t\Vol[u](s)\diff s
-\sum_{j=-N}^N \Vol[u](\tDE_j)\wDE_j(t)
\right]
+\left[
\sum_{j=-N}^N \left\{\Vol[u](\tDE_j) - \VolDE[u](\tDE_j) \right\}\wDE_j(t)
\right].
\end{align*}
Since Theorem~\ref{thm:rewrite-DE-assump}
claims
$\mu u Q$ and $(\Vol u)Q$
belong to $\LC_{\alpha}(\DEt(\domD_d))$,
the first and second terms are bounded 
using Theorem~\ref{Thm:DE-Sinc-Indef}
as in~\eqref{leq:G-DE-indef}.
For the third term, since
$k(s,\cdot)u(\cdot)Q(\cdot)\in\LC_{\alpha}(\DEt(\domD_d))$
from Theorem~\ref{thm:rewrite-DE-assump},
$|\Vol[u](\tDE_j) - \VolDE[u](\tDE_j)|$
is bounded using Theorem~\ref{Thm:DE-Sinc-Indef}.
Using the bound, we have
\begin{align*}
\left|
\sum_{j=-N}^N \left\{\Vol[u](\tDE_j) - \VolDE[u](\tDE_j) \right\}
\cdot \wDE_j(t)
\right|
\leq \sum_{j=-N}^N
\left(
\tilde{C}_5 \frac{\log(2 d N/\alpha)}{N}
\rme^{-\pi d N/\log(2 d N/\alpha)} \right)
\cdot |\wDE_j(t)|
\end{align*}
for some constant $\tilde{C}_5$.
The final task is to bound $\sum_{j=-N}^N|\wDE_j(t)|$.
Since
$h\sum_{j=-N}^N\DEtDiv(jh)$ converges to $(b-a)$
as $N\to\infty$,
and from Lemma~\ref{Lem:Bound-J-Real},
there exists a constant $\tilde{C}_6$ such that
\[
\sum_{j=-N}^N|\wDE_j(t)|
\leq \sum_{j=-N}^N (1.1 h \DEtDiv(jh))
\leq \tilde{C}_6.
\]
\end{proof}

\section{Numerical examples}
\label{sec:numer-exam}

In this section, we present
numerical results that support the convergence theorems.
All computation programs were written in C++
with double-precision floating-point arithmetic.
The sine integral $\Si(x)$ is computed using
the routine in the GNU Scientific Library.
When checking the assumptions
of Theorems~\ref{thm:SE-converge} and~\ref{thm:DE-converge},
$\epsilon$ is used as an arbitrary small positive number.

In the first example, all functions in the equation
are entire functions.
\begin{example}
\label{exmp1}
Consider the following equation~\cite[Example 3.2]{1986Brunner}
\begin{equation*}
 u'(t)=1 + 2t - u(t) + \int_0^t t(1+2t)\rme^{r(t-r)}u(r)\diff r,
\quad 0\leq t\leq 1,
\end{equation*}
with $u(0)=1$. The exact solution is $u(t)=\rme^{t^2}$.
\end{example}
In the SE case,
the assumptions in
Theorem~\ref{thm:SE-converge}
are fulfilled with $\alpha=1$ (note that $g$, $\mu$, and $k$ is bounded)
and $d=\pi - \epsilon \simeq 3.14$ (note that $d<\pi$).
In the DE case,
the assumptions in Theorem~\ref{thm:DE-converge}
are fulfilled with $\alpha=1$ (as in the SE case)
and $d=\pi/2 - \epsilon \simeq 1.57$ (note that $d<\pi/2$).
The schemes were implemented with these values for parameters $\alpha$ and $d$.
The errors were investigated
on $999$ equally spaced points in $[0,\,1]$,
and their maximum is indicated in Figure~\ref{fig:exmp1}
by the label ``\texttt{maximum error}.''
We can observe the theoretical rates;
$\Order(\exp(-c\sqrt{N}))$ in the SE-Sinc-Nystr\"{o}m method,
and
$\Order(\exp(-c' N/\log N))$ in the DE-Sinc-Nystr\"{o}m method.
Both methods converge exponentially,
but DE's rate is much higher than SE's rate.

In the next example, there is a pole at $t=-1$,
which affects the DE case.
\begin{example}
\label{exmp2}
Consider the following equation~\cite[Example 3]{2010Zarebnia}
\begin{equation*}
 u'(t)=
\frac{1}{1+t}-
\frac{\left(2+t\log(1+t)\right)\log(1+t)}{2}
+u(t)+\int_0^t \frac{t}{r+1}u(r)\diff r,
\quad 0\leq t\leq 1,
\end{equation*}
with $u(0)=0$. The exact solution is $u(t)=\log(1+t)$.
\end{example}
In the SE case,
the assumptions in
Theorem~\ref{thm:SE-converge}
are fulfilled with $\alpha=1$ and
$d=\pi - \epsilon \simeq 3.14$
(as in Example~\ref{exmp1}).
In the DE case,
we define $Z$ as
\[
Z=\frac{\pi}{\log 2}
\sqrt{\frac{2}{1+\sqrt{1+(2\pi/\log 2)^2}}}.
\]
Then, the assumptions in Theorem~\ref{thm:DE-converge}
are fulfilled with $\alpha=1$ (as in Example~\ref{exmp1})
and $d=\arctan(Z) - \epsilon \simeq 1.11$
(note that $g(\DEt(\zeta))$ and $k(z,\DEt(\zeta))$ is
not analytic at $\zeta=\arcsinh(\pm\imnum - \log(2)/\pi)$).
The errors were investigated in the same way as in Example~\ref{exmp1},
and are shown in Figure~\ref{fig:exmp2}.
The graph shows the theoretical rates.

\begin{figure}[htbp]
\centering
\begin{minipage}{0.48\linewidth}
\includegraphics[width=\linewidth]{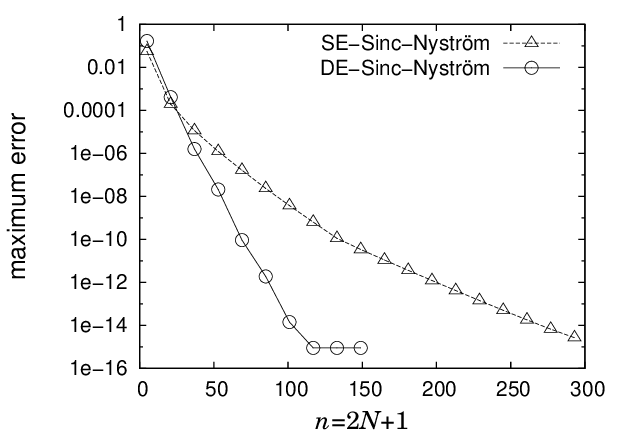}
\caption{Errors in Example~\ref{exmp1}.}
\label{fig:exmp1}
\end{minipage}
\begin{minipage}{0.48\linewidth}
\includegraphics[width=\linewidth]{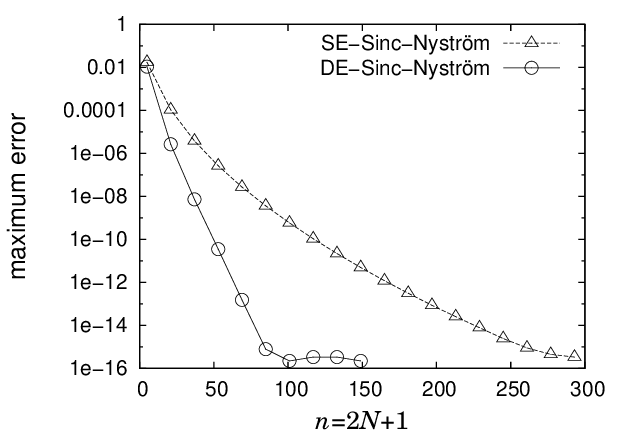}
\caption{Errors in Example~\ref{exmp2}.}
\label{fig:exmp2}
\end{minipage}
\end{figure}

The next example is more difficult because of a weak singularity at the origin.
\begin{example}
\label{exmp3}
Consider the following equation
\begin{equation*}
 u'(t)=
\frac{1}{2\sqrt{t}}
- t u(t)+\int_0^t \sqrt{\frac{t}{r}} u(r)\diff r,
\quad 0\leq t\leq 1,
\end{equation*}
with $u(0)=0$. The exact solution is $u(t)=\sqrt{t}$.
\end{example}
In the SE case,
the assumptions in
Theorem~\ref{thm:SE-converge}
are fulfilled with $\alpha=1/2$ (note that $g(t)Q(t)=t^{1/2}(1-t)/2\in\LC_{1/2}(\domD)$)
and $d=\pi - \epsilon \simeq 3.14$ (as in Example~\ref{exmp1}).
In the DE case,
the assumptions in Theorem~\ref{thm:DE-converge}
are fulfilled with $\alpha=1/2$ (as in the SE case)
and $d=\pi/2 - \epsilon \simeq 1.57$ (as in Example~\ref{exmp1}).
The errors are shown in Figure~\ref{fig:exmp3},
which shows the theoretical rates in this case as well.

The next example is even more difficult because of the infinite singular points
distributed around the endpoints.
\begin{example}
\label{exmp4}
Let $p(t)=\sin(4\arctanh t)$ and $q(t)=\cos(4\arctanh t)+\cosh(\pi)$,
and consider the following equation
\begin{align*}
 u'(t)&=
- t\sqrt{\frac{q(t)}{1-t^2}}
- \frac{2 p(t)}{\sqrt{(1-t^2)q(t)}}
+\sqrt{(3+t^2)(1-t^2)} u(t)\\
&\quad + \int_0^t 2\sqrt{\frac{3+t^2}{1-r^2}}
\left\{r + \frac{p(r)}{q(r)}\right\} u(r)\diff r,
\quad -1\leq t\leq 1,
\end{align*}
with $u(-1)=0$. The exact solution is $u(t)=\sqrt{(1-t^2)q(t)}$.
\end{example}
In the SE case,
the assumptions in
Theorem~\ref{thm:SE-converge}
are fulfilled with $\alpha=1/2$
and $d=\pi/2 - \epsilon \simeq 1.57$ (as in Example~\ref{exmp3}).
In contrast, in the DE case,
the assumptions in Theorem~\ref{thm:DE-converge}
are not fulfilled for any $d>0$ (although $\alpha=1/2$ can be found
as in Example~\ref{exmp3}),
and we do not expect to attain $\Order(\exp(-c'N/\log N))$.
However, according to Tanaka et al.~\cite{tanaka13:_desinc},
the DE-Sinc indefinite integration still converges
with a rate similar to that of SE if we
set $d=\arcsin((\pi/2 - \epsilon)/\pi)\simeq 0.523$.
The errors are shown in Figure~\ref{fig:exmp4},
and the two methods converge at similar rates.

\begin{figure}[htbp]
\centering
\begin{minipage}{0.48\linewidth}
\includegraphics[width=\linewidth]{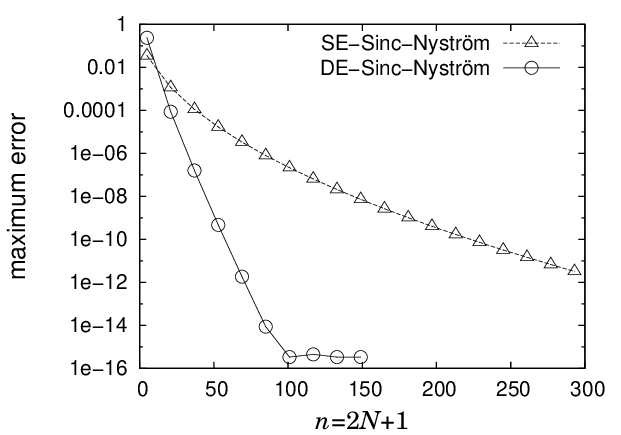}
\caption{Errors in Example~\ref{exmp3}.}
\label{fig:exmp3}
\end{minipage}
\begin{minipage}{0.48\linewidth}
\includegraphics[width=\linewidth]{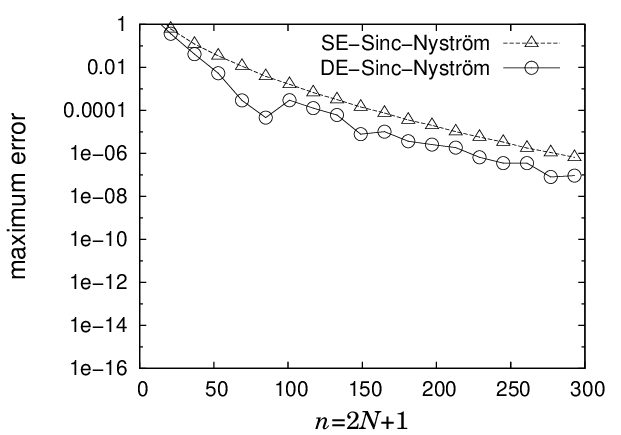}
\caption{Errors in Example~\ref{exmp4}.}
\label{fig:exmp4}
\end{minipage}
\end{figure}

In the final example, similar to Example~\ref{exmp1},
all functions in the equation are entire functions.

\begin{example}
\label{exmp5}
Consider the following equation~\cite{lin00:_petrov}
\begin{align*}
 u'(t)&=\cos t
+ \frac{\rme^t}{5}\left\{ \rme^{2t}(\cos t - 2\sin t)-1\right\}
 +\int_0^t \rme^{t+2r}u(r)\diff r,
\quad 0\leq t\leq 1,
\end{align*}
with $u(0)=0$. The exact solution is $u(t)=\sin t$.
\end{example}

The assumptions in
Theorems~\ref{thm:SE-converge}
and~\ref{thm:DE-converge}
are fulfilled with the same $\alpha$ and $d$ as those of Example~\ref{exmp1}.
The results of the SE-Sinc-Nystr\"{o}m method
and DE-Sinc-Nystr\"{o}m method are shown in Figure~\ref{fig:exmp5},
with the results of the postprocessing PGFE method~\cite{lin00:_petrov}.
Note that the horizontal axis in Figure~\ref{fig:exmp5}
is a logarithmic scale axis.
The convergence rate of the postprocessing PGFE method
is polynomial: $\Order(N^{-4})$,
whereas those of the two Sinc-Nystr\"{o}m methods are exponential.
For this reason, the two Sinc-Nystr\"{o}m methods
eventually overtake the postprocessing PGFE method.

\begin{figure}[htbp]
\centering
\begin{minipage}{0.48\linewidth}
\includegraphics[width=\linewidth]{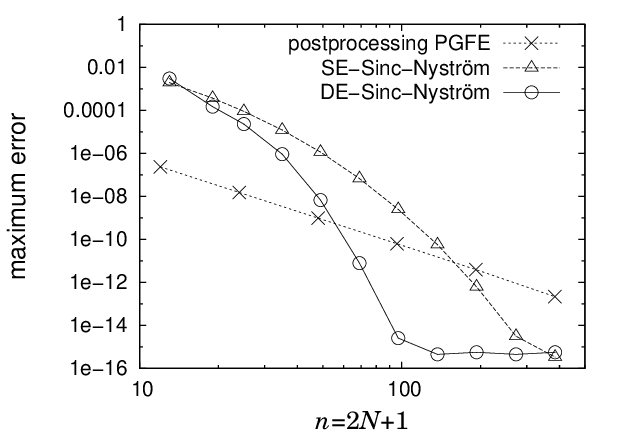}
\caption{Errors in Example~\ref{exmp5}.}
\label{fig:exmp5}
\end{minipage}
\end{figure}

\section{Concluding remarks}
\label{sec:concl}

A Sinc-Nystr\"{o}m method for~\eqref{eq:Volterra-integro-diff}
was developed by Zarebnia~\cite{2010Zarebnia}
(called the SE-Sinc-Nystr\"{o}m method in this paper),
for which there remain two points to be discussed.
First, the convergence rate of the method was suggested
as $\Order(\exp(-c\sqrt{N}))$, but not proved.
Second, the regularity of solution $u$ is necessary for
implementation, although $u$ is an unknown function to be determined.
For the first point, this paper showed by theoretical analysis that
the convergence rate is in fact $\Order(\exp(-c\sqrt{N}))$.
For the second point, this paper showed by theoretical analysis that
the regularity of solution $u$ can be determined from
the known functions $g$, $\mu$, and $k$.

In addition, this paper proposed a new method
called the DE-Sinc-Nystr\"{o}m method
by replacing the variable transformation in
the SE-Sinc-Nystr\"{o}m method.
By a theoretical analysis, this paper showed that
$\Order(\exp(-c'N/\log N))$ can be attained by
the DE-Sinc-Nystr\"{o}m method,
and also showed the same result as above on the second point
(the regularity of $u$).

As explained in Remark~\ref{rem:condition-number},
the invertibility of
the coefficient matrix of the resulting linear system was proved in this paper,
but uniform-boundedness of the norm of the matrix was not proved.
The latter point
will be investigated on another occasion.
Generalization of the presented methods for nonlinear
Volterra integro-differential equations is also considered
as a future work.
\bibliography{x}
\end{document}